\documentclass{amsart}

\usepackage{enumerate}
\usepackage{amssymb}
\usepackage{graphicx}
\usepackage{MnSymbol}

\usepackage{amsthm}

\title[Nunke type classification]{A Nunke type classification in the locally compact setting}

\author{Samuel M. Corson}

\author{Olga Varghese}

\theoremstyle{definition}\newtheorem{theorem}{Theorem}

\theoremstyle{definition}

\theoremstyle{definition}\newtheorem{corollary}[theorem]{Corollary}
\theoremstyle{definition}\newtheorem{proposition}[theorem]{Proposition}
\theoremstyle{definition}
\theoremstyle{definition}
\theoremstyle{definition}
\theoremstyle{definition}\newtheorem{remark}[theorem]{Remark}
\theoremstyle{definition}
\theoremstyle{definition}\newtheorem{lemma}[theorem]{Lemma}
\theoremstyle{definition}
\theoremstyle{definition}
\theoremstyle{definition}
\theoremstyle{definition}
\theoremstyle{definition}
\theoremstyle{definition}

\begin{document}

\address{Instituto de Ciencias Matem\'aticas CSIC-UAM-UC3M-UCM, 28049 Madrid, Spain.}
\email{sammyc973@gmail.com}

\address{Department of Mathematics, M\"unster University, Einsteinstra{\upshape{\ss}}e 62, 48149, M\"unster, Germany}.

\email{olga.varghese@uni-muenster.de}

\thanks{The work of the first author was supported by the Severo Ochoa Programme for Centres of Excellence in R\&D SEV-20150554.  The work of the second author was funded by the Deutsche Forschungsgemeinschaft (DFG, German Research Foundation) under Germany's Excellence Strategy EXC 2044–390685587, Mathematics M\"nster: Dynamics-Geometry-Structure.}

\begin{abstract} In this short note we prove that a group $G$ is lcH-slender- that is, every abstract group homomorphism from a locally compact Hausdorff topological group to $G$ has an open kernel- if and only if $G$ is torsion-free and does not include $\mathbb{Q}$ or the p-adic integers $\mathbb{Z}_p$ for any prime $p$.  This mirrors a classical characterization given by Nunke for slender abelian groups. 
\end{abstract}

\maketitle

\begin{section}{Introduction}

R. Nunke produced in 1961 a remarkable theorem which comprehensively characterizes a special class of abelian groups via their subgroups \cite{Nu}.  An abelian group $A$ is \emph{slender} if for every abstract group homomorphism $\phi$ whose domain is the countably infinite product $\prod_{\omega}\mathbb{Z}$ and whose codomain is $A$ there exists an $m\in \omega$ such that $\phi = \phi\circ p_m$, where $p_m: \prod_{\omega}\mathbb{Z} \rightarrow \prod_{k=0}^m\mathbb{Z} \times (0)_{k=m+1}^{\infty}$ is the obvious retraction map \cite{Fu}.  Thus $A$ is slender provided each homomorphism $\phi:\prod_{\omega}\mathbb{Z} \rightarrow A$ depends on only finitely many coordinates, or equivalently, provided any such $\phi$ has open kernel (endowing $\prod_{\omega}\mathbb{Z}$ with the Tychonov topology where each coordinate is discrete).  Nunke's theorem is that an abelian group $A$ is slender if and only if $A$ is torsion-free and does not include $\prod_{\omega}\mathbb{Z}$ or $\mathbb{Q}$ or a p-adic integer group $\mathbb{Z}_p$ for any prime $p$.

Slenderness can be seen as an automatic continuity condition: endowing an abelian slender group $A$ a discrete topology one sees that any homomorphism from $\prod_{\omega}\mathbb{Z}$ to $A$ is continuous.  By analogy, one defines a (not necessarily abelian) group $G$ to be \emph{locally compact Hausdorff slender} (abbrev. \emph{lcH-slender}) provided every abstract group homomorphism from a locally compact Hausdorff topological group to $G$ has open kernel \cite{ConCor}.  An early result of Dudley shows that, for example, free (abelian) groups are lcH-slender \cite{D}.  More recent results extend the known lcH-slender groups to include each group whose abelian subgroups are free \cite{KV}.

We deduce the following complete classification of lcH-slender groups via their subgroups:

\begin{theorem}\label{maintheorem}  A group $G$ is lcH-slender if and only if $G$ is torsion-free and does not include $\mathbb{Q}$ or any p-adic integer group $\mathbb{Z}_p$ as a subgroup.
\end{theorem}

This classification was recently shown in the case where $G$ is abelian \cite{C}.  Theorem \ref{maintheorem} provides the following further examples of lcH-slender groups (using \cite[Theorem I.4.1 (vii)]{Da} and \cite[Theorem 1]{N}, respectively):

\begin{corollary}\label{niceclasses}  Torsion-free CAT(0) groups and torsion-free one-relator groups are lcH-slender.
\end{corollary}

An analogous notion of slenderness is defined by replacing the locally compact Hausdorff groups with the completely metrizable groups: $G$ is \emph{completely metrizable slender} (abbrev. \emph{cm-slender}) if every homomorphism from a completely metrizable topological group to $G$ has open kernel.  Similarly $G$ is \emph{n-slender} if every abstract group homomorphism from the fundamental group of a Peano continuum to $G$ has an open kernel.  As the proscribed subgroups in Theorem \ref{maintheorem} are neither cm-slender nor n-slender (see \cite[Theorem C]{C} and \cite[Theorem 3.3]{E}) we imediately obtain:

\begin{corollary}\label{slendernessimplications}  If a group $G$ is either n-slender or cm-slender then $G$ is lcH-slender.
\end{corollary}

\end{section}

\begin{section}{Proof of Theorem \ref{maintheorem}}\label{proof}

Towards the proof of Theorem \ref{maintheorem} we recall some notions from abelian group theory (see \cite{Fu}).  We use convenient alternative characterizations which suit our purposes, rather than the historical definitions.  Let $A$ be an abelian group.  We say $A$ is \emph{algebraically compact} if $A$ is an abstract direct summand of a Hausdorff compact abelian group.  We say $A$ is \emph{cotorsion} if it is a homomorphic image of an algebraically compact group.  Also, $A$ is \emph{cotorsion-free} if the only cotorsion subgroup of $A$ is the trivial one.  Importantly $A$  is cotorsion-free if and only if $A$ is torsion-free and contains no copy of $\mathbb{Q}$ or the p-adic integers $\mathbb{Z}_p$ for any prime $p$  \cite[Theorem 13.3.8]{Fu}.

\begin{lemma}\label{compactcase}  Suppose $\phi: K \rightarrow G$ is an abstract group homomorphism with $K$ a compact topological group and $G$ torsion-free and not including $\mathbb{Q}$ or any $\mathbb{Z}_p$ as a subgroup.  Then $\phi$ is trivial.
\end{lemma}

\begin{proof}   Let $k\in K$ be given.  It is easy to verify that the closure $\overline{\langle k\rangle}\leq K$ of the cyclic subgroup generated by $k$ is compact abelian.  Then $\phi(\overline{\langle k\rangle})$ is abelian and cotorsion, on the one hand, but cotorsion-free on the other hand by our assumptions on $G$.  Therefore $\phi(\overline{\langle k\rangle})$ is trivial and the lemma is proved.
\end{proof}

Next we state some classical results regarding locally compact groups.

\begin{proposition}  Let $L$ be a locally compact Hausdorff group.

\begin{enumerate}

\item (Iwasawa's structure Theorem, \cite[Theorem 13]{I}) If $L$ is connected then we can write $L = H_0\cdots H_j K$ where each $H_i$ is a subgroup of $L$ isomorphic to $\mathbb{R}$ and $K$ is a compact subgroup of $L$.

\item (van Dantzig's Theorem, \cite[III \S 4, No. 6]{B}) If $L$ is locally compact Hausdorff and totally disconnected then $L$ has a compact open subgroup.

\end{enumerate}

\end{proposition}

\begin{proof}[Proof of Theorem \ref{maintheorem}]

If $G$ is lcH-slender it is certainly necessary that $G$ is torsion-free and not include $\mathbb{Q}$ or any $\mathbb{Z}_p$.  If, for example, $G$ were to include torsion then $G$ would contain a subgroup of prime order $p$ and one can construct a discontinuous homomorphism $\phi: \prod_{\omega}\mathbb{Z}/p\mathbb{Z} \rightarrow G$ using a vector space argument.  Similarly if $G$ includes $\mathbb{Q}$ as a subgroup one produces a discontinuous homomorphism $\phi: \mathbb{R} \rightarrow G$ by selecting a Hamel basis for $\mathbb{R}$.  Were $G$ to include $\mathbb{Z}_p$ as a subgroup then the inclusion map $\mathbb{Z}_p \rightarrow G$ witnesses that $G$ is not lcH-slender.

For the other implication we suppose that $G$ is torsion-free and does not include $\mathbb{Q}$ or any $\mathbb{Z}_p$ and let $\phi:L \rightarrow G$ be an abstract group homomorphism with $L$ a locally compact Hausdorff group.  Let $L^\circ$ denote the connected component of the identity element.  Since $L^\circ$ is closed, it is itself locally compact Hausdorff, and as $L^\circ$ is also connected we know by Iwasawa's structure Theorem that $L^\circ = H_0\cdots H_j K$ with each $H_i$ a subgroup isomorphic to $\mathbb{R}$ and $K$ a compact subgroup.  We know that $\phi\upharpoonright K$ is trivial by Lemma \ref{compactcase} and since $G$ is torsion-free and includes no subgroup isomorphic to $\mathbb{Q}$ we know that $\phi\upharpoonright H_i$ is trivial for each $i$.  Thus $\phi(L^\circ)$ is trivial.

Now the homomorphism $\phi: L \rightarrow G$ passes to a homomorphism $\overline{\phi}: L/L^\circ \rightarrow G$.  By van Dantzig's Theorem we have a compact open subgroup $K' \leq L/L^\circ$.  Again by Lemma \ref{compactcase} we know that $\overline{\phi}\upharpoonright K'$ is trivial, and thus $\pi^{-1}(K') \leq \ker(\phi)$ witnesses that $\ker(\phi)$ is open, where $\pi: L \rightarrow L/L^\circ$ is the continuous projection.
\end{proof}

\begin{remark}  A nice subgroup characterization for n-slender and cm-slender groups is still apparently beyond reach.  Thus far a classification exists for such groups only in the abelian case \cite{C}.
\end{remark}

\end{section}

%\section*{Acknowledgement}

\end{document}